\documentclass[12pt]{article}
\usepackage{amsmath,amsfonts,amssymb,amsthm,amscd}
\title{Affine Nash groups over real closed fields}
\date{April 28, 2011}
\author{Ehud Hrushovski\thanks{Supported by ISF}\\Hebrew University of Jerusalem \and Anand Pillay\thanks{Supported by EPSRC grant EP/I002294/1}\\University of Leeds}
\newtheorem{Theorem}{Theorem}[section]
\newtheorem{Proposition}[Theorem]{Proposition}
\newtheorem{Definition}[Theorem]{Definition} 
\newtheorem{Remark}[Theorem]{Remark}
\newtheorem{Lemma}[Theorem]{Lemma}
\newtheorem{Corollary}[Theorem]{Corollary}

\newcommand{\R}{\mathbb R}

\newcommand{\C}{\mathbb C}

\begin{document}
\maketitle

\begin{abstract} 
We prove that a semialgebraically connected affine Nash group over a real closed field $R$ is Nash isogenous to the semialgebraically connected component of the group $H(R)$ of $R$-points of some algebraic group $H$ defined over $R$. In the case when $R = \R$ this result was claimed in \cite{HP}, but a mistake in the proof was recently found, and the new proof we obtained has the advantage of being valid over an arbitrary real closed field.
We also extend the result to not necessarily connected affine Nash groups over arbitrary real closed fields.
\end{abstract}

\section{Introduction and preliminaries}
The Nash category lies in between the real algebraic and real analytic categories. Nash functions are by definition both semialgebraic and analytic. In so far as Nash manifolds are concerned, the relevant transition functions should be Nash, and one also requires a finite covering by charts which are open semialgebraic subsets of $\R^{n}$.  A Nash manifold is said to be {\em affine} if it is Nash embeddable in some $\R^{n}$. If $X$ is the set of real points of some nonsingular quasiprojective algebraic variety defined over $\R$ then $X$ is an affine Nash manifold. Conversely the {\em algebraicity theorem} (see Proposition 8.4.6 of \cite{BCR}) says that any affine Nash manifold is Nash isomorphic to a connected component of a nonsingular real algebraic variety. There is a considerable literature on affine Nash manifolds but less on (abstract) Nash manifolds, although the latter were also considered in the pioneering paper of Artin and Mazur \cite{AM} and later in Shiota's comprehensive monograph \cite{Shiota}. 

A Nash group is a Nash manifold with Nash group structure. In \cite{HP} we purported to prove an ``equivariant" version of the algebraicity theorem by showing that a connected affine Nash group is Nash isogeneous to the connected component of a real algebraic group. Recently a mistake was pointed out by Elias Baro. We are not sure whether the proof in \cite{HP} can be salvaged, but we will give here a somewhat different and more direct proof, which also works over arbitrary real closed fields. This is done in section 2. In section 3 we give a version for not necessarily connected (affine) Nash groups: if $G$ is a Nash group whose connected component is affine then $G$ is Nash isogenous with a subgroup of finite index in a real algebraic group. We also carry this out at the level of real closed fields.

We would first like to thank Elias Baro who pointed out the precise mistake in \cite{HP} as well as suggesting that our new proof should work over any real closed field. In this connection we should also mention Margarita Otero who also had pointed out some gaps in the proof in \cite{HP} which until recently we thought we could fill easily. Secondly, thanks to Sun Binyong,  for asking us about the generalization to the case of non connected affine Nash groups and for his commentary on the proof described to him by the first author.

We now give some definitions and facts, with \cite{BCR} as a basic reference. We will work over an arbitrary real closed field $R$, a special case being when $R = \R$.  
 If $U$ is an open semialgebraic subset of $R$ then by a Nash function $f:U \to R$ we mean a function which is semialgebraic and infinitely differentiable, in the obvious sense. When $R = \R$ this amounts to $f$ being (real) analytic and satisfying a polynomial equation 
$P({\bar x},f({\bar x})) = 0$ on $U$. The definition of an abstract Nash manifold over $R$, of Nash maps between such Nash manifolds, and hence of a Nash group over $\R$ is unproblematic. But there does not seem to be a systematic treatment of ``differential geometric" properties of such abstract Nash manifolds when $R$ is not the reals. So we will take as our {\em definition} of an ``affine Nash manifold" that of a 
{\em $d$-dimensional Nash submanifold $M$ of $R^{n}$} from 
\cite{BCR}  (Definition 2.9.9). $M$ should be a semialgebraic subset of $R^{n}$ with the following property: for every $a\in M$ there is an open semialgebraic neighbourhood $U$ of $0$ in $R^{n}$ and and open semialgebraic neighbourhood $V$ of $a$ in $R^{n}$, and a Nash diffeomorphism 
$\phi$ between $U$ and $V$ such that $\phi((R^{d}\times \{0\})\cap U) = M \cap V$. 

A Nash function or mapping $f$ from $M$ to $R$ is by definition a semialgebraic function such that for every $\phi$ as above the map $f\circ\phi$ restricted to $R^{d}\times\{0\})\cap U$ is Nash  (considered as a mapping from a semialgebraic open subset of $R^{d}$ to $R$). Note that in particular the coordinate functions on $M$ are Nash. 
We deduce easily the notion of a Nash mapping from $M$ to $N$ where $M, N$ are affine Nash manifolds.  

A Nash submanifold $M$ of $R^{n}$ has a topology induced from $R^{n}$ and we call it {\em semialgebraically connected} if we cannot write $M$ as the disjoint union of two nonempty open semialgebraic subsets. In general an affine Nash manifold is the disjoint union of finitely many definably connected components, each of which is also an affine Nash manifold.

By an {\em affine Nash group $G$} we mean an affine Nash manifold with a group structure such that the multiplication and inversion maps, from $G\times G \to G$ and $G\to G$ respectively are Nash maps. If $G$ is an affine Nash group then $G^{0}$ denotes the semialgebraically connected component of the identity, also an affine Nash group. 

The notions of a {\em real algebraic variety} and {\em affine real algebraic variety} over $R$ as well as regular maps between them are discussed in detail in section 3 of \cite{BCR}. We will call these $R$-algebraic varieties and affine $R$-algebraic varieties. An $R$-algebraic group is an $R$-algebraic variety with group structure (product, inversion) given by regular maps.
A key difference with usual algebraic geometry (over an algebraically closed field) is that projective $n$-space over $R$, ${\mathbb P}_{n}(R)$, is (biregularly isomorphic to) an affine $R$-algebraic variety (Theorem 3.4.4 of \cite{BCR}). A consequence is that if $X$ is a quasiprojective algebraic variety over $R$ (or defined over $R$) in the usual sense, then the set $X(R)$ of $R$-points of $X$ is (naturally) an affine $R$-algebraic variety. On the face of it $X(R)$ need not be Zariski dense in $X$, but replacing $X$ by the Zariski closure of $X(R)$ in $X$, we can always assume Zariski-density of $X(R)$. As algebraic groups are quasiprojective  it follows that if $G$ is an algebraic group defined over $R$ then $G(R)$ is an affine $R$-algebraic group, in particular an affine Nash group over $R$. 

If $H$ is an $\R$-algebraic group then $H^{0}$ as well as finite covers of $H$, will be affine Nash groups but not necessarily $\R$-algebraic groups. Likewise over any real closed field $R$. So the most one can expect to prove is that an affine Nash group $G$ is  Nash isogenous to a union of semialgebraic connected components of an $R$-algebraic group (which is what we prove). Moreover it suffices to prove that there is a Nash homomorphism $f$ with finite kernel from $G$ into some $R$-algebraic group $H$, because then $f(G)$ will have finite index in its Zariski closure (an $R$-algebraic
subgroup of $H$). 

Although the role of model theoretic ideas in this paper is somewhat suppressed, we say a few words. A {\em semialgebraic} subset $X$ of $R^{n}$ is the same thing as a set which is (first order) definable (with parameters) in the structure
$(R,+,\cdot, <)$  (because of quantifier elimination). By a semialgebraic function $f$ between semialgebraic sets we mean simply a function whose graph is semialgebraic, so we make no continuity assumption. (But any semialgebraic function is ``piecewise Nash".) By a semialgebraic group we mean a semialgebraic set with semialgebraic group operation. 
Nash groups (affine or otherwise) are semialgebraic groups and moreover any semialgebraic homomorphism between Nash groups is Nash. A result in \cite{Pillay} (together with the above-mentioned piecewise Nashness of semialgebraic functions) implies that any semialgebraic group is semialgebraically isomorphic (as a group)  to a (not necessarily affine) Nash group. However there are many semialgebraic groups which are not semialgebraically isomorphic to {\em affine} Nash groups: Fix $a>0$ in $R$ then the semialgebraic group with universe the interval $[0,a)$ and with group operation addition modulo $a$ is a well-known such example. A rather new kind of example appears in \cite{CPI}, see 2.10 there. It seems not too unreasonable at the current time to aim towards a fairly explicit description of all semialgebraic groups over real closed fields (up to semialgebraic isomorphism), starting from the $R$-algebraic groups and iterating some basic constructions.

\vspace{2mm}
\noindent
We will be making use of the notion of the {\em dimension} of a semialgebraic set $X\subseteq R^{n}$. See section 2.8 of \cite{BCR} or any model theory text such as \cite{vdD}.

\begin{Definition} Let $M$ be an affine Nash manifold. By a Nash subset of $M$ we mean the common zero set of finitely many Nash functions $f:M\to R$. 
\end{Definition}

So a Nash subset of $M$ is a special case of a semialgebraic subset of $M$.

\vspace{2mm}
\noindent
An important fact for us will be:
\begin{Lemma} Let $M$ be an an affine Nash manifold. Then we have the descending chain condition on Nash subsets of $M$: there is no infinite strictly descending chain $X_{1}\supset X_{2} \supset ...$ of Nash subsets of $M$.
\end{Lemma}
\begin{proof}  This is given by Proposition 8.6.2 of \cite{BCR} that the common zero set in $M$ of an ideal $I$ in the ring of Nash functions on $M$, is the common zero set of finitely many functions in $I$.
\end{proof}

\begin{Remark} Let $M\subset R^{n}$ be an affine Nash manifold of dimension $d$. Then the Zariski closure of $M$ in $R^{n}$ has dimension $d$.
\end{Remark}

\section{Algebraicity of semialgebraically connected affine Nash groups}
Remember that $R$ denotes an arbitrary real closed field.

We prove:
\begin{Theorem} Let $G$ be a semialgebraically connected affine Nash group over $R$. Then $G$ is Nash isogenous to the semialgebraic connected component of the identity of some $R$-algebraic group $H(R)$.  Equivalent there is a Nash homomorphism with finite kernel of $G$ into some $R$-algebraic group $H(R)$.
\end{Theorem}

The strategy is as in the purported proof in \cite{HP} in the case $R = \R$.  
\newline
Step I, which does not make use of affineness, is to find a (connected) algebraic group $H$ over $R$ and a local Nash isomorphism between $G$ and $H(R)$.  By a local Nash isomorphism between $G$ and $H(R)$ we mean some Nash diffeomorphism $f$ between open semialgebraic neighbourhoods of the identity, $U$, $V$ of $G$, $H(R)$ respectively, such that for any $g,h\in U$, if $g\cdot h \in U$, then $f(g\cdot h) = f(g)\cdot f(h)$.

This is precisely Theorem A from \cite{HP} the proof of which goes through with no change for an arbitrary real closed field. Model theory, in the guise of the first author's group configuration techniques, played a role in the proof. The 
reader is referred to section 3, as well as Lemmas 4.8 and Corollary 4.9 of \cite{HP}.

 Step II replaces the ``proof of Theorem B" in section 4 of \cite{HP} which made use of universal covers of real Lie groups, but had a mistake (proof of the CLAIM on p. 240, in which we implicitly assumed that the image of a discrete subgroup under a continuous homomorphism of Lie group is also discrete). 
 
 Let $f$ be the local Nash isomorphism between $G$ and $H(R)$ given by Step I. Note that $G\times H(R)$ is an affine Nash group (in particular an affine Nash manifold). By Lemma 1.2 we have the $DCC$ on Nash subsets of $G\times H(R)$, so in particular any subset $Y$ of $G\times H(R)$ has a ``Nash closure": smallest Nash subset of $G\times H(R)$ containing $Y$.
For each open semialgebraic neighbourhood $U$ of the identity of $G$
 contained in $dom(f)$, let $A_{U}\subset G\times H(R)$ be the graph of the restriction of $f$ to $U$, and let $B_{U}\subset G\times H(R)$ be the Nash closure of $A_{U}$.

By Lemma 1.2 again  $B = \cap_{U}B_{U}$ is a finite subintersection, and so of the form $B_{U_{0}}$ for fixed $U_{0}$ which we may assume to be symmetric (i.e. $U_{0} = U_{0}^{-1}$).

\begin{Lemma} $dim(B) = d  = dim(G) (= dim(H(R))$. 
\end{Lemma}
\begin{proof}  Note that for each $U$ the dimension of $A_{U}$ is $d$, and hence by 1.3 the the dimension of the Zariski closure of $A_{U}$ is $d$. As
 $A_{U} \subset B_{U} \subseteq$ Zariski closure of $A_{U}$, it follows that $dim(B_{U}) = d$.
 \end{proof}

\begin{Lemma} $B$ is a subgroup of $G\times H(R)$.
\end{Lemma}
\begin{proof} Let $U_{1}\subseteq U_{0}$ be a symmetric neighbourhood of the identity in $G$ such that $U_{1}\cdot U_{1}\subseteq U_{0}$. We now work in the group $G\times H(R)$.
\newline
{\em Claim 1.} For any $a\in A_{U_{1}}$, and $x\in B$, $a\cdot x \in B$.
\newline
{\em Proof.} Note that $X_{a} = \{x\in G\times H(R): a\cdot x \in B\}$ is a Nash subset of $G\times H(R)$  (as $B$ is a Nash subset and the group operation is Nash). But if $x\in A_{U_{1}}$ then $a\cdot x \in A_{U_{0}} \subset B_{U_{0}} = B$. Hence $X_{a}$ contains $A_{U_{1}}$. But  $B_{U_{1}} = B_{U_{0}} = B$, so $X_{a}$ contains $B$ as required. 

\vspace{2mm}
\noindent
Using Claim 1 we obtain in a similar fashion:
\newline
{\em Claim 2.} For any $a\in B$ and $x\in B$, $x\cdot a  \in B$.

\vspace{2mm}
\noindent
Likewise we see that $B$ is closed under inversion. Hence the Lemma is proved.

\end{proof}

So $B$ is a semialgebraic subgroup of $G\times H(R)$ of dimension $d$ ($=dim(G) = dim(H(R))$). It follows that $B^{0}$ is a semialgebraic subgroup of $G\times H(R)^{0}$ of dimension $d$, projecting onto each factor. The  ``kernel" of $B$, $\{g\in G: (g,e) \in B\}$ and ``cokernel" of $B$, 
$\{h\in H(R)^{0}: (e,h)\in B\}$ are clearly finite, normal (so central) subgroups of $G$, $H(R)^{0}$ respectively. Let $C$ denote the cokernel. As we may assume $H(R)^{0}$ to be Zariski dense in the connected algebraic group $H$, $C$ is also a finite central subgroup of $H$. 
Then $H/C$ is a connected algebraic group defined over $R$, and we have a semialgebraic isomorphism between $(H/C)(R)^{0}$ and $H(R)^{0}/C$. 
So the isogeny from $G$ onto $H(R)^{0}/C$ can be identified with a (Nash) isogeny from $G$ onto the (semialgebraic) connected component of $(H/C)(R)$.
The proof of 2.1 is complete.

\section{The non connected case}

This section is devoted to a proof of:
\begin{Proposition}  Let $G$ be an arbitrary (not necessarily semialgebraically connected) affine Nash group over $R$. Then there is a Nash homomorphism with finite kernel from $G$ into some $R$-algebraic group $H(R)$.
\end{Proposition}

Let $G^{0}$ be the semialgebraically connected component of $G$. All we will use is that $G^{0}$ is affine Nash. 
Note that we are free to replace $G$ by $G/N$ for $N$ any finite normal subgroup. Hence by Theorem 2.1 we may assume that  $G^{0} = H_{1}(R)^{0}$ where $H_{1}$ is a connected algebraic group defined over $R$. For $g\in G$, 
let $\alpha_{g}: G^{0} \to G^{0}$ be conjugation by $g$. The basic idea is to extend each $\alpha_{g}$ to an $R$-rational automorphism $\beta_{g}$ of $H_{1}$, at the expense of replacing $H_{1}$ by an isogenous algebraic group defined over $R$. It will then be easy to construct the required algebraic group $H$ (whose connected component is $H_{1}$). 

We go through various steps. First we may and will assume that $R$ is of cardinality continuum, by passing to an elementary extension or substructure of $R$. Hence $R(i)$ is an algebraically closed field of cardinality the continuum, which we can assume to be the field $\C$ of complex numbers. We identify the algebraic group $H_{1}$ with its group of complex points $H_{1}(\C)$. Let $\pi: \widetilde{H_{1}}\to H_{1}$ be the universal cover of $H_{1}$, a simply connected complex Lie group, and $\Gamma$ the kernel of $\pi$, a central discrete subgroup.  For $g\in G$, let $K_{g}$ be the Zariski closure in $H_{1}\times H_{1}$ of the graph of $\alpha_{g}$. It is not hard to see that $K_{g}$ is a connected algebraic subgroup of
$H_{1}\times H_{1}$ with finite-to-one projections on each coordinate, and with both ``kernel" and "cokernel" being central. Let $\widetilde{K_{g}}$ be the universal cover of $K_{g}$, also a complex Lie group. 

\begin{Lemma} (i) $\widetilde{K_{g}}$ naturally identifies with the graph of an automorphism $\widetilde{\alpha_{g}}$ of $\widetilde{H_{1}}$, which lifts $\alpha_{g}$.
\newline
(ii) $\widetilde{\alpha_{g}}(\Gamma) \cap \Gamma$ has finite index in each of $\widetilde{\alpha_{g}}(\Gamma)$ and $\Gamma$.
\newline
(iii) When $g\in G^{0}$, $\widetilde{\alpha_{g}}$ acts trivially on $Z(\widetilde{H_{1}})$, in particular acts trivially on $\widetilde{\alpha_{h}}(\Gamma)$ for all $h\in G$.
\end{Lemma}
\begin{proof} (i) The coordinate projections $p_{1}, p_{2} :K_{g} \to H_{1}$ lift to  (analytic) isomorphisms  $p_{1}', p_{2}'$ between 
$\widetilde{K_{g}}$ and $\widetilde{H_{1}}$, whereby $\widetilde {K_{g}}$ is the graph of an analytic automorphism of $\widetilde{H_{1}}$.

\vspace{2mm}
\noindent
(ii) By considering $\widetilde{\alpha_{g^{-1}}}$ it suffices to prove that for each $g$, 
\newline
(*) $\widetilde{\alpha_{g}}(\Gamma)\cap \Gamma$ has finite index in $\widetilde{\alpha_{g}}(\Gamma)$. 
\newline
Now $K_{g}\subset H_{1}\times H_{1}$ has finite cokernel enumerated by $\overline d$ say. Let ${\overline d}^{\prime}$ be a lifting of the tuple 
$\overline d$ to a tuple in $\widetilde{H_{1}}$. Let $a\in \Gamma$. Then $\widetilde{\alpha_{g}}(a)\in \Gamma\cdot {\overline d}^{\prime}$. Hence 
$\widetilde{\alpha_{g}}(\Gamma)$ meets only finitely many translates of $\Gamma$ in $\widetilde{H_{1}}$, yielding (*).

\vspace{2mm}
\noindent
(iii) If $g\in G^{0}$, then $\alpha_{g}$ is already a rational map, so the Zariski closure $K_{g}$ of its graph is still conjugation by $g$ in $H_{1}$. It is easy to see that $\widetilde{\alpha_{g}}$ is then conjugation in $\widetilde{H_{1}}$ by some (any) lift $\widetilde g$ of $g$, hence acts trivially on $Z(\widetilde{H_{1}})$.
\end{proof}

\begin{Corollary}  Let $\Gamma'$ be the subgroup of $\widetilde{H_{1}}$ generated by the set of $\widetilde{\alpha_{g}}(\Gamma)$ for $g\in G$. Then 
$\Gamma$ is a subgroup of $\Gamma'$ of finite index.
\end{Corollary}
\begin{proof} Clearly $\Gamma\subset \Gamma'$, as by Lemma 3.2(iii) $\widetilde{\alpha_{g}}(\Gamma) = \Gamma$ for $g\in G^{0}$.
Also by 3.2(iii) again, $\widetilde{\alpha_{g}}(\Gamma)$ depends only on the coset of $g$ modulo $G^{0}$. So, as $G^{0}$ has finite index in $G$, 
$\{\widetilde{\alpha_{g}}(\Gamma):g\in G\}$ is finite, so by Lemma 3.2(ii), $\Gamma$ has finite index in $\Gamma'$.
\end{proof}

By the Corollary $N = \Gamma'/\Gamma$ is a finite (central) subgroup of $H_{1}$, and the quotient of $H_{1}$ by $N$ is a connected algebraic group $H_{2}$, say (which also equals  $\widetilde{H_{1}}/\Gamma'$). Let $\tau:H_{1} \to H_{2}$ be the canonical surjective homomorphism.

\begin{Lemma} (i) For each $g\in G$, the automorphism $\widetilde{\alpha_{g}}$ of $\widetilde{H_{1}}$ induces an  automorphism $\beta_{g}$ of 
$H_{2} = \widetilde{H_{1}}/\Gamma'$. 
\newline
(ii) For any $g\in G$, $\beta_{g}$ is a rational automorphism of the algebraic group $H_{2}$, and $(\beta_{g}\circ \tau)|G^{0} = 
(\tau \circ \alpha_{g})|G^{0}$.
\end{Lemma}
\begin{proof} (i) is clear as $\Gamma'$ is invariant under each $\widetilde{\alpha_{g}}$.
\newline
(ii) The graph of $\beta_{g}$ is clearly the image of $K_{g}$ under the projection $\tau \times \tau: H_{1} \times H_{1} \to H\times H$, hence $\beta$ is a rational automorphism of $H$. The second part follows as $K_{g}|(G^{0}\times G^{0})$ is the graph of $\alpha_{g}$
\end{proof}

We now use the $\beta_{g}$'s to build a (complex) algebraic group $H$ whose connected component is $H_{2}$.   Let $g_{1},..,g_{n}$ be representatives of the cosets of $G^{0}$ in $G$. Let $g_{i}\cdot g_{j} = h_{ij}\cdot g_{r}$  where $r = r(i,j)$ and $h_{ij} \in G^{0}$. 
Note that the group $(G,\cdot)$ is isomorphic to  the set $G^{0}\times\{g_{1},..,g_{n}\}$ equipped with the group operation $(h,g_{i})*(h',g_{j}) = 
(h\cdot (\alpha_{g_{i}}(h'))\cdot h_{ij},g_{r(i,j)})$, via the map taking $h\cdot g_{i}$ to $(h,g_{i})$. 
Let us now {\em define} the group $H$ to be the set $H_{2}\times \{g_{1},..,g_{n}\}$  with group operation  $(h,g_{i})*(h',g_{j}) =
(h\cdot (\beta_{g_{i}}(h'))\cdot \tau(h_{ij}), g_{r(i,j)})$.   Note that $H$ is definable with parameters in $(\C,+,\cdot)$, hence is definably isomorphic to a complex algebraic group (whose connected component is clearly $H_{2}$). 

Let us now map $G$ to $H$ by the map  $f(h\cdot g_{i}) = (h,g_{i})$, and we note that this is a homomorphism with finite (central) kernel (contained in $G^{0}$). 

At this point we can conclude the proof of Proposition 3.1 in two possible ways.
\newline
(A): We have already said that $H$ can be identified with a complex algebraic group. Identifying $\C$ with $R\times R$, we identify $H$ with an $R$-algebraic group, and we note that $f$ is semialgebraic (hence Nash)

\vspace{2mm}
\noindent
(B):  Alternatively we can analyse $H_{2}$ and the construction of $H$ to see that $H$ is actually a (complex) algebraic group defined over $R$ and that the homomorphism $f:G\to H$ is also defined over $R$, whereby $f:G\to H(R)$ is the required semialgebraic (in fact Nash) homomorphism with finite 
central kernel.

The proof of Proposition 3.1 is complete.

\end{document}